\documentclass[10pt]{amsart}
\usepackage{amssymb,MnSymbol}
\usepackage{amsthm,amsmath}

\title[Measuring the interactions among variables]{Measuring the interactions among variables of functions over the unit hypercube}

\author{Jean-Luc Marichal}
\address{Mathematics Research Unit, FSTC, University of Luxembourg, 6, rue Coudenhove-Kalergi, L-1359 Luxembourg, Grand Duchy of Luxembourg.}
\email{jean-luc.marichal[at]uni.lu}

\author{Pierre Mathonet}
\address{Mathematics Research Unit, FSTC, University of Luxembourg, 6, rue Coudenhove-Kalergi, L-1359 Luxembourg, Grand Duchy of Luxembourg.}
\email{pierre.mathonet[at]uni.lu}

\date{February 21, 2011}

\begin{document}

\theoremstyle{plain}
\newtheorem{theorem}{Theorem}
\newtheorem{lemma}[theorem]{Lemma}
\newtheorem{proposition}[theorem]{Proposition}
\newtheorem{corollary}[theorem]{Corollary}
\newtheorem{fact}[theorem]{Fact}
\newtheorem*{main}{Main Theorem}

\theoremstyle{definition}
\newtheorem{definition}[theorem]{Definition}
\newtheorem{example}[theorem]{Example}

\theoremstyle{remark}
\newtheorem*{conjecture}{\indent Conjecture}
\newtheorem{remark}{Remark}
\newtheorem{claim}{Claim}

\newcommand{\N}{\mathbb{N}}
\newcommand{\R}{\mathbb{R}}
\newcommand{\I}{\mathbb{I}}
\newcommand{\Vspace}{\vspace{2ex}}
\newcommand{\bfx}{\mathbf{x}}
\newcommand{\bfy}{\mathbf{y}}
\newcommand{\bfh}{\mathbf{h}}
\newcommand{\bfe}{\mathbf{e}}
\newcommand{\Q}{Q}

\begin{abstract}
By considering a least squares approximation of a given square integrable function $f\colon[0,1]^n\to\R$ by a multilinear polynomial of a specified degree, we define an index which measures the overall interaction among variables of $f$. This definition extends the concept of Banzhaf interaction index introduced in cooperative game theory. Our approach is partly inspired from multilinear regression analysis, where interactions among the independent variables are taken into consideration. We show that this interaction index has appealing properties which naturally generalize several properties of the Banzhaf interaction index. In particular, we interpret this index as an expected value of the difference quotients of $f$ or, under certain natural conditions on $f$, as an expected value of the derivatives of $f$. Finally, we discuss a few applications of the interaction index in aggregation function theory.
\end{abstract}

\keywords{Interaction index, multilinear polynomial, least squares approximation, difference operator, aggregation function, cooperative game, fuzzy game}

\subjclass[2010]{Primary 26B35, 41A10, 93E24; Secondary 39A70, 91A12}

\maketitle

\section{Introduction}

Sophisticated mathematical models are extensively used in a variety of areas of mathematics and physics, and especially in applied
fields such as engineering, life sciences, economics, finance, and many others. Here we consider the simple
situation where the model aims at explaining a single dependent variable, call it $y$, in terms of $n$ independent variables $x_1,\ldots,x_n$.
Such a model is usually described through an equation of the form
$$
y=f(x_1,\ldots,x_n),
$$
where $f$ is a real function of $n$ variables.

Now, suppose that the function $f$ describing the model is given and that we want to investigate its behavior through simple terms. For instance, suppose we want to measure the overall contribution (importance or influence) of each independent variable to the model. A natural approach to this problem consists in defining the overall importance of each variable as the coefficient of this variable in the least squares linear approximation of $f$. This approach was considered
by Hammer and Holzman~\cite{HamHol92} for pseudo-Boolean functions and cooperative games $f\colon\{0,1\}^n\to\R$. Interestingly enough, they observed
that the coefficient of each variable in the linear approximation is exactly the Banzhaf power index \cite{Ban65,DubSha79} of the corresponding player in
the game $f$.

In many practical situations, the information provided by the overall importance degree of each variable may be far insufficient due to the
possible interactions among the variables. Then, a more flexible approach to investigate the behavior of $f$ consists in measuring an overall importance degree for
each combination (subset) of variables. Such a concept was first introduced in \cite{KahKalLin88} for Boolean functions $f\colon\{0,1\}^n\to\{0,1\}$ (see also \cite{BenLin90,BouKahKalKatLin92}), then in \cite{Mar00} for pseudo-Boolean functions and games $f\colon\{0,1\}^n\to\R$ (see also \cite{MarKojFuj07}), and in \cite{GraLab01} for square
integrable functions $f\colon [0,1]^n\to\R$.

In addition to these importance indexes, we can also measure directly the interaction degree among the variables by defining an overall interaction index for each combination of variables. This concept was introduced axiomatically in \cite{GraRou99} (see also \cite{FujKojMar06}) for games $f\colon\{0,1\}^n\to\R$. However, it has not yet been extended to real functions defined on $[0,1]^n$, even though such functions are of growing importance for instance in aggregation function theory. In this paper we intend to fill this gap by defining and investigating an appropriate index to measure the interaction degree among variables of a given square
integrable function $f\colon [0,1]^n\to\R$.

Our sources of inspiration to define such an index are actually threefold:
\begin{description}
 \item[In cooperative game theory] Interaction indexes were introduced axiomatically a decade ago \cite{GraRou99} for games $f\colon\{0,1\}^n\to\R$ (see also \cite{FujKojMar06}). The best known interaction indexes are the Banzhaf and Shapley interaction indexes, which extend the Banzhaf and Shapley power indexes. Following Hammer and Holzman's approach~\cite{HamHol92}, it was shown in \cite{GraMarRou00} that the Banzhaf interaction index can be obtained from least squares approximations of the game under consideration by games whose multilinear representations are of lower degrees.

\item[In analysis] Considering a sufficiently differentiable real function $f$ of several variables, the \emph{local} interaction among certain variables at a given point $\mathbf{a}$ can be obtained through the coefficients of the Taylor expansion of $f$ at $\mathbf{a}$, that is, through the coefficients of the \emph{local} polynomial approximation of $f$ at $\mathbf{a}$. By contrast, if we want to define an \emph{overall} interaction index, we naturally have to consider a \emph{global} approximation of $f$ by a polynomial function. 

\item[In statistics] Multilinear statistical models have been proposed to take into account the interaction among the independent variables (see for instance \cite{AikWes91}): two-way interactions appear as the coefficients of leading terms in quadratic models, three-way interactions appear as the coefficients of leading terms in cubic models, and so forth.
\end{description}

On the basis of these observations, we naturally consider the least squares approximation problem of a given square integrable function $f\colon [0,1]^n\to\R$ by a polynomial of a given degree. As multiple occurrences in combinations of variables are not relevant, we will only consider multilinear polynomial functions. Then, given a subset $S\subseteq\{1,\ldots,n\}$, an index ${\mathcal I}(f,S)$ measuring the interaction among the variables $\{x_i:i\in S\}$ of $f$ is defined as the coefficient of the monomial $\prod_{i\in S}x_i$ in the best approximation of $f$ by a multilinear polynomial of degree at most $|S|$. This definition is given and discussed in Section 2, where we also provide an interpretation in the context of cooperative fuzzy games (Remark~\ref{rem:s87df6}).

In Section 3 we show that this new index has many appealing properties, such as linearity, continuity, and symmetry. In particular, we show that, similarly to the Banzhaf interaction index introduced for games, the index ${\mathcal I}(f,S)$ can be interpreted in a sense as an expected value of the discrete derivative of $f$ in the direction of $S$ (Theorem~\ref{thm:Labreuche}) or, equivalently, as an expected value of the difference quotient of $f$ in the direction of $S$ (Corollary~\ref{cor:sfadfsd}). Under certain natural conditions on $f$, the index can also be interpreted as an expected value of the derivative of $f$ in the direction of $S$ (Proposition~\ref{Beta1}). These latter results reveal a strong analogy between the interaction index and the overall importance index introduced by Grabisch and Labreuche~\cite{GraLab01}.

In Section 4 we discuss certain applications in aggregation function theory, including the computation of explicit expressions of the interaction index for the discrete Choquet integrals. We also define and investigate a normalized version of the interaction index to compare different functions in terms of interaction degrees of their variables and a coefficient of determination to measure the quality of multilinear approximations.

We employ the following notation throughout the paper. Let $\I^n$ denote the $n$-dimensional unit cube $[0,1]^n$. We denote by $F(\I^n)$ the class of all functions $f\colon \I^n\to\R$ and by $L^2(\I^n)$ the class of square integrable functions $f\colon \I^n\to\R$ (modulo equality almost everywhere). For any $S\subseteq N=\{1,\ldots,n\}$, we denote by $\mathbf{1}_S$ the characteristic vector of $S$ in $\{0,1\}^n$.

\section{Interaction indexes}

In this section we first recall the concepts of power and interaction indexes introduced in cooperative game theory and how the Banzhaf index can be obtained from the solution of a least squares approximation problem. Then we show how this approximation problem can be extended to functions in $L^2(\I^n)$ and, from this extension, we introduce an interaction index for such functions.

Recall that a (\emph{cooperative})~\emph{game} on a finite set of players $N=\{1,\ldots,n\}$ is a set function $v\colon 2^N\to\R$ which assigns to each coalition $S$ of players a real number $v(S)$ representing the \emph{worth} of $S$.\footnote{Usually, the condition $v(\varnothing)=0$ is required for $v$ to define a game. However, we do not need this restriction in the present paper.} Through the usual identification of the subsets of $N$ with the elements of $\{0,1\}^n$, a game $v\colon 2^N\to\R$ can be equivalently described by a pseudo-Boolean function $f\colon\{0,1\}^n\to\R$. The correspondence is given by $v(S)=f(\mathbf{1}_S)$ and
\begin{equation}\label{eq:pBfPF}
f(\bfx)=\sum_{S\subseteq N} v(S)\,\prod_{i\in S}x_i\,\prod_{i\in N\setminus S}(1-x_i).
\end{equation}
Equation~(\ref{eq:pBfPF}) shows that any pseudo-Boolean function $f\colon\{0,1\}^n\to\R$ can always be represented by a multilinear polynomial of degree at
most $n$ (see \cite{HamRud68}), which can be further simplified into
\begin{equation}\label{eq:fMob}
f(\bfx)=\sum_{S\subseteq N} a(S)\,\prod_{i\in S}x_i\, ,
\end{equation}
where the set function $a\colon 2^N\to\R$, called the \emph{M\"obius transform} of $v$, is defined by
$$
a(S)=\sum_{T\subseteq S} (-1)^{|S|-|T|}\, v(T).
$$

Let $\mathcal{G}^N$ denote the set of games on $N$. A \emph{power index} \cite{Sha53} on $N$ is a function $\phi\colon\mathcal{G}^N\times N\to\R$
that assigns to every player $i\in N$ in a game $f\in\mathcal{G}^N$ his/her prospect $\phi(f,i)$ from playing the game.
An \emph{interaction index} \cite{GraRou99} on $N$ is a function $I\colon\mathcal{G}^N\times 2^N\to\R$ that
measures in a game $f\in\mathcal{G}^N$ the interaction degree among the players of a coalition $S\subseteq N$.

For instance, the \emph{Banzhaf
interaction index} \cite{GraRou99} of a coalition $S\subseteq N$ in a game $f\in\mathcal{G}^N$ is defined (in terms of the M\"obius transform of $f$) by
\begin{equation}\label{eq:IBI23}
I_\mathrm{B}(f,S)=\sum_{T\supseteq S} \Big(\frac 12\Big)^{|T|-|S|} a(T),
\end{equation}
and the \emph{Banzhaf power index} \cite{DubSha79} of a player $i\in N$ in a game $f\in\mathcal{G}^N$ is defined by $\phi_\mathrm{B}(f,i)=I_\mathrm{B}(f,\{i\})$.

It is noteworthy that $I_\mathrm{B}(f,S)$ can be interpreted as an average of the \emph{$S$-difference} (or \emph{discrete $S$-derivative}) $\Delta^Sf$ of $f$. Indeed, it can be shown (see \cite[{\S}2]{GraMarRou00}) that
\begin{equation}\label{DiscBanzhaf}
I_\mathrm{B}(f,S)=\frac 1{2^n}\sum_{\bfx\in\{0,1\}^n}(\Delta^Sf)(\bfx),
\end{equation}
where $\Delta^Sf$ is defined inductively by $\Delta^{\varnothing}f=f$ and $\Delta^Sf=\Delta^{\{i\}}\Delta^{S\setminus\{i\}}f$ for $i\in S$, with $\Delta^{\{i\}}f(\mathbf{x})=f(\mathbf{x}\mid x_i=1)-f(\mathbf{x}\mid x_i=0)$.

We now recall how the Banzhaf interaction index can be obtained from a least squares approximation problem. For $k\in\{0,\ldots,n\}$, denote by $V_k$ the set of all multilinear polynomials $g\colon\{0,1\}^n\to\R$ of degree at most $k$, that is of the form
\begin{equation}\label{eq:GinVK}
g(\bfx)=\sum_{\textstyle{S\subseteq N\atop |S|\leqslant k}} c(S)\prod_{i\in S}x_i\, ,
\end{equation}
where the coefficients $c(S)$ are real numbers. For a given pseudo-Boolean function $f\colon\{0,1\}^n\to\R$, the best $k$th approximation of $f$ is the unique multilinear
polynomial $f_k\in V_k$ that minimizes the squared distance
$$
\sum_{\bfx\in\{0,1\}^n}(f(\bfx)-g(\bfx))^2
$$
among all $g\in V_k$. A closed-form expression of
$f_k$ was given in \cite{HamHol92} for $k=1$ and $k=2$ and in \cite{GraMarRou00} for arbitrary $k\leqslant n$. In fact, when $f$ is given in its multilinear form (\ref{eq:fMob}) we obtain
$$
f_k(\bfx)=\sum_{\textstyle{S\subseteq N\atop |S|\leqslant k}} a_k(S)\prod_{i\in S}x_i,
$$
where
$$
a_k(S)=a(S)+(-1)^{k-|S|}\sum_{\textstyle{T\supseteq S\atop |T|>k}}{|T|-|S|-1\choose k-|S|}\,\Big(\frac 12\Big)^{|T|-|S|}a(T).
$$
It is then easy to see that
\begin{equation}\label{shgfsdf}
I_\mathrm{B}(f,S)=a_{|S|}(S).
\end{equation}
Thus,
$I_\mathrm{B}(f,S)$ is exactly the coefficient of the monomial $\prod_{i\in S}x_i$ in the best approximation of $f$ by a multilinear polynomial of degree at most $|S|$.

Taking into account this approximation problem, we now define an interaction index for functions in $L^2(\I^n)$ as follows. Denote by $W_k$ the set of all multilinear polynomials $g\colon\I^n\to\R$ of
degree at most $k$. Clearly, these functions are also of the form (\ref{eq:GinVK}). For a given function $f\in L^2(\I^n)$, we define the
\emph{best $k$th (multilinear) approximation of $f$} as the multilinear polynomial $f_k\in W_k$ that minimizes the squared distance
\begin{equation}\label{eq:dist-f-g}
\int_{\I^n}\big(f(\bfx)-g(\bfx)\big)^2\, d\bfx
\end{equation}
among all $g\in W_k$.

It is easy to see that $W_k$ is a linear subspace of $L^2(\I^n)$ of dimension $\sum_{s=0}^k {n\choose s}$. Indeed, $W_k$ is the linear span of
the basis $B_k=\{v_S : S\subseteq N, \, |S|\leqslant k\}$, where the functions $v_S\colon\I^n\to\R$ are defined by $v_S(\bfx)=\prod_{i\in
S}x_i$. Note that formula (\ref{eq:dist-f-g}) also writes $\lVert f-g\rVert^2$ where $\|\cdot\|$ is the standard norm of $L^2(\I^n)$ associated with the inner product $\langle f,g\rangle=\int_{\I^n}f(\bfx)g(\bfx)\, d\bfx$. Therefore, using the general theory of Hilbert spaces, the solution of this approximation problem
exists and is uniquely determined by the orthogonal projection of $f$ onto $W_k$. This projection can be easily expressed in any orthonormal basis of $W_k$. But here it is very easy to see that the set $B'_k=\{w_S : S\subseteq N, \, |S|\leqslant k\}$, where $w_S\colon\I^n\to\R$ is given by
$$
w_S(\bfx)=12^{|S|/2}\prod_{i\in S}\Big(x_i-\frac 12\Big)
=12^{|S|/2}\,\sum_{T\subseteq S}\Big(-\frac 12\Big)^{|S|-|T|} v_T(\bfx),
$$
forms such an orthonormal basis for $W_k$.

The following immediate theorem gives the components of the best $k$th approximation of a function $f\in L^2(\I^n)$ in the bases $B_k$ and
$B'_k$.
\begin{theorem}\label{thm:Coeffk}
For every $k\in\{0,\ldots,n\}$, the best $k$th approximation of $f\in L^2(\I^n)$ is the function
\begin{equation}\label{eq:akS2}
f_k=\sum_{\textstyle{T\subseteq N\atop |T|\leqslant k}}\langle f,w_T\rangle \, w_T ~=~ \sum_{\textstyle{S\subseteq N\atop |S|\leqslant k}}
a_k(S)\, v_S\, ,
\end{equation}
where
\begin{equation}\label{eq:akS}
a_k(S)=\sum_{\textstyle{T\supseteq S\atop |T|\leqslant k}} \Big(\! -\frac 12\Big)^{|T|-|S|} 12^{|T|/2}\, \langle f,w_T\rangle.
\end{equation}
\end{theorem}

By analogy with (\ref{shgfsdf}), to measure the interaction degree among variables of an arbitrary function $f\in L^2(\I^n)$, we naturally define an index $\mathcal{I}\colon L^2(\I^n)\times
2^{N}\to\R$ as $\mathcal{I}(f,S)=a_{|S|}(S)$, where $a_{|S|}(S)$ is obtained from $f$ by (\ref{eq:akS}). We will see in the next section that this index indeed measures an importance degree when $|S|=1$ and an interaction degree when $|S|\geqslant 2$.

\begin{definition}
Let $\mathcal{I}\colon L^2(\I^n)\times 2^{N}\to\R$ be defined as $\mathcal{I}(f,S)=12^{|S|/2}\langle f,w_S\rangle$, that is,
\begin{equation}\label{eq:GenBI1}
\mathcal{I}(f,S)= 12^{|S|}\int_{\I^n}f(\bfx)\, \prod_{i\in S}\Big(x_i-\frac 12\Big)\, d\bfx.
\end{equation}
\end{definition}

Thus we have defined an interaction index from an approximation (projection) problem. Conversely, this index characterizes this
approximation problem. Indeed, as the following result shows, the best $k$th approximation of $f\in L^2(\I^n)$ is the unique function of $W_k$ that preserves the interaction index for all the $s$-subsets such that $s\leqslant k$. The discrete analogue of this result was established in
\cite{GraMarRou00} for the Banzhaf interaction index (\ref{eq:IBI23}).

\begin{proposition}\label{prop:Proj-IB}
A function $f_k\in W_k$ is the best $k$th approximation of $f\in L^2(\I^n)$ if and only if
$\mathcal{I}(f,S)=\mathcal{I}(f_k,S)$ for all $S\subseteq N$ such that $|S|\leqslant k$.
\end{proposition}

\begin{proof}
By definition, we have $\mathcal{I}(f,S)=\mathcal{I}(f_k,S)$ if and only if $\langle f-f_k,w_S\rangle =0$ for all
$S\subseteq N$ such that $|S|\leqslant k$, and the latter condition characterizes the projection of $f$ onto $W_k$.
\end{proof}

The explicit conversion formulas between the interaction index and the best approximation can be easily derived from the preceding results. On the one hand, by
(\ref{eq:akS}), we have
$$
a_k(S)=\sum_{\textstyle{T\supseteq S\atop |T|\leqslant k}}\Big(\! -\frac 12\Big)^{|T|-|S|}\, \mathcal{I}(f,T)\, ,\qquad
\mbox{for}~|S|\leqslant k.
$$
On the other hand, by Proposition~\ref{prop:Proj-IB} and Equation (\ref{eq:akS2}), we also have
\begin{eqnarray*}
\mathcal{I}(f,S) &=& \mathcal{I}(f_k,S)=12^{|S|/2}\,\langle f_k,w_S\rangle \\
&=& 12^{|S|/2}\sum_{\textstyle{T\subseteq N\atop |T|\leqslant k}}a_k(T)\,\langle v_T,w_S\rangle
\end{eqnarray*}
that is, by calculating the inner product,
\begin{equation}\label{eq:sfhksd45}
\mathcal{I}(f,S)=\sum_{\textstyle{T\supseteq S\atop |T|\leqslant k}}\Big(\frac 12\Big)^{|T|-|S|}\, a_k(T)\, ,\qquad
\mbox{for}~|S|\leqslant k.
\end{equation}

We also note that, by (\ref{eq:akS2}), the best $k$th approximation of $f$ can be expressed in terms of $\mathcal{I}$ as
\begin{equation}\label{eq:Approx}
f_k(\bfx) = \sum_{\textstyle{T\subseteq N\atop |T|\leqslant k}} \mathcal{I}(f,T)\, \prod_{i\in T}\Big(x_i-\frac 12\Big).
\end{equation}
Using the notation $\boldsymbol{\frac 12}=(\frac 12,\ldots,\frac 12)$, the Taylor expansion formula then shows that
$$
\mbox{$\mathcal{I}(f,S)=(D^S f_k)(\boldsymbol{\frac 12})\, ,\qquad$ for $|S|\leqslant k$,}
$$
where $D^S$ stands for the partial derivative operator with respect to the variables $x_i$ for $i\in S$. In particular,
$\mathcal{I}(f,\varnothing)=\int_{\I^n} f(\bfx)\, d\bfx = f_k(\textstyle{\boldsymbol{\frac 12}})$.

We also have the following result, which shows that the index $\mathcal{I}$ generalizes the Banzhaf interaction index $I_\mathrm{B}$. First note that the restriction operation $f\mapsto f|_{\{0,1\}^n}$ defines a linear bijection between the spaces $W_n$ and $V_n$. The inverse map is the so-called ``multilinear extension''.

\begin{proposition}
For every function $f\in W_n$ and every subset $S\subseteq N$, we have $\mathcal{I}(f,S)=I_\mathrm{B}(f|_{\{0,1\}^n},S)$.
\end{proposition}

\begin{proof}
Let $f\in W_n$ of the form $f(\mathbf{x})=\sum_{T\subseteq N}a(T)\prod_{i\in T}x_i$ and let $S\subseteq N$. Then, using (\ref{eq:sfhksd45}) for $k=n$ and recalling that $a(T)=a_n(T)$ for every $T\subseteq N$, we obtain
$$
\mathcal{I}(f,S)=\sum_{T\supseteq S}\Big(\frac 12\Big)^{|T|-|S|}\, a(T).
$$
We then conclude by formula (\ref{eq:IBI23}).
\end{proof}

\begin{remark}\label{rem:s87df6}
In cooperative game theory, the set $F(\I^n)$ can be interpreted as the set of \emph{fuzzy games} (see for instance Aubin~\cite{Aub81}). In this context, a \emph{fuzzy coalition} is simply an element $\bfx\in\I^n$ and a \emph{fuzzy game} $f\in F(\I^n)$ is a mapping that associates with any fuzzy coalition its \emph{worth}. It is now clear that the index $\mathcal{I}$ is a natural extension of the Banzhaf interaction index to fuzzy games in $L^2(\I^n)$ when this index is regarded as a solution of a multilinear approximation problem.
\end{remark}

\section{Properties and interpretations}

Most of the interaction indexes defined for games, including the Banzhaf interaction index, share a set of fundamental properties such as linearity, symmetry, and monotonicity (see \cite{FujKojMar06}). Many of them can also be expressed as expected values of the discrete derivatives (differences) of their arguments (see for instance (\ref{DiscBanzhaf})). In this section we show that the index $\mathcal{I}$ fulfills direct generalizations of these properties to the framework of functions of $L^2(\I^n)$. In particular, we show that $\mathcal{I}(f,S)$ can be interpreted as an expected value of the difference quotient of $f$ in the direction of $S$ or, under certain natural conditions on $f$, as an expected value of the derivative of $f$ in the direction of $S$.

The first result follows from the very definition of the index.

\begin{proposition}\label{prop:lin56}
For every $S\subseteq N$, the mapping $f\mapsto\mathcal{I}(f,S)$ is linear and continuous.
\end{proposition}

Recall that if $\pi$ is a permutation on $N$, then, for every function $f\in F(\I^n)$, the permutation $\pi$ acts on $f$ by
$
\pi(f)(x_1,\ldots,x_n)=f(x_{\pi(1)},\ldots,x_{\pi(n)}).
$
The following result is then an easy consequence of the change of variables theorem.
\begin{proposition}
The index $\mathcal{I}$ is symmetric. That is, for every permutation $\pi$ on $N$, every $f\in L^2(\I^n)$, and every $S\subseteq N$, we have $\mathcal{I}(\pi(f),\pi(S))=\mathcal{I}(f,S)$.
\end{proposition}

We now provide an interpretation of $\mathcal{I}(f,S)$ as an expected value of the $S$-derivative $D^Sf$ of $f$. The proof immediately follows from repeated integrations by parts of (\ref{eq:GenBI1}) and thus is omitted.

For $S\subseteq N$, denote by $h_S$ the probability density function of independent beta distributions on $\I^n$ with parameters $\alpha =\beta =2$, that is, $h_S(\bfx)=6^{|S|}\prod_{i\in S}x_i(1-x_i)$.

\begin{proposition}\label{Beta1}
For every $S\subseteq N$ and every $f\in L^2(\I^n)$ such that $D^Tf$ is continuous and integrable on $]0,1[^n$ for all $T\subseteq S$, we have
\begin{equation}\label{eq:BetaDer}
\mathcal{I}(f,S)=\int_{\I^n}h_S(\bfx)\, D^Sf(\bfx)\, d\bfx.
\end{equation}
\end{proposition}

\begin{remark}\label{rem:11}
\begin{enumerate}

\item[(a)] Formulas~(\ref{DiscBanzhaf}) and (\ref{eq:BetaDer}) show a strong analogy between the indexes $I_\mathrm{B}$ and $\mathcal{I}$. Indeed, $I_\mathrm{B}(f,S)$ is the expected value of the $S$-difference of $f$ with respect to the discrete uniform distribution whereas $\mathcal{I}(f,S)$ is the expected value of the $S$-derivative of $f$ with respect to a beta distribution. We will see in Theorem~\ref{thm:Labreuche} a similar interpretation of $\mathcal{I}(f,S)$ which does not require all the assumptions of Proposition~\ref{Beta1}.

\item[(b)] Propositions~\ref{prop:Proj-IB} and \ref{Beta1} reveal an analogy between least squares approximations and Taylor expansion formula. Indeed, while the $k$-degree Taylor expansion of $f$ at a given point $\mathbf{a}$ can be seen as the unique polynomial of degree at most $k$ whose derivatives at $\mathbf{a}$ coincide with the derivatives of $f$ at the same point, the best $k$th approximation of $f$ is the unique multilinear polynomial of degree at most $k$ that agrees with $f$ in all average $S$-derivatives for $|S|\leqslant k$.
\end{enumerate}
\end{remark}

We now give an alternative interpretation of $\mathcal{I}(f,S)$ as an expected value, which does not require the additional assumptions of Proposition~\ref{Beta1}. In this more general framework, we naturally replace the derivative with a difference quotient. To this extent, we introduce some further notation. As usual, we denote by $\bfe_i$ the $i$th vector of the standard basis of $\R^n$. For every $S\subseteq N$ and every $\bfh\in \I^n$, we define the \emph{$S$-shift} operator $E^S_\bfh$ on $F(\I^n)$ by
\[
E^S_\bfh f(\bfx)=f\bigg(\bfx+\sum_{j\in S}h_j \bfe_j\bigg)
\]
for every $\mathbf{x}\in\I^n$ such that $\mathbf{x}+\mathbf{h}\in\I^n$.

We also define the \emph{$S$-difference} (or \emph{discrete $S$-derivative}) operator $\Delta^S_\bfh$ on $F(\I^n)$ inductively by $\Delta^{\varnothing}_\bfh f=f$ and $\Delta^S_\bfh f=\Delta^{\{i\}}_\bfh\Delta^{S\setminus\{i\}}_\bfh f$ for $i\in S$, with
$
\Delta^{\{i\}}_\bfh f(\bfx)=E^{\{i\}}_\bfh f(\bfx)-f(\bfx).
$
Similarly, we define the \emph{$S$-difference quotient} operator $Q^S_\bfh$ on $F(\I^n)$ by $Q^{\varnothing}_\bfh f=f$ and $Q^S_\bfh f=Q^{\{i\}}_\bfh Q^{S\setminus\{i\}}_\bfh f$ for $i\in S$, with
$
Q^{\{i\}}_\bfh f(\bfx)=\frac{1}{h_i}\Delta^{\{i\}}_\bfh f(\bfx).
$

The next straightforward lemma provides a direct link between the difference operators and the shift operators. It actually shows that, for every fixed $\bfh\in\I^n$, the map $S\mapsto \Delta_{\bfh}^S$ is nothing other than the M\"obius transform of the map $S\mapsto E^S_\bfh$.

\begin{lemma}\label{lemmadelta2}
For every $f\in F(\I^n)$ and every $S\subseteq N$, we have
\begin{equation}\label{eq:Int10}
\Delta_{\bfh}^Sf(\bfx)=\sum_{T\subseteq S}(-1)^{|S|-|T|}\, E^T_\bfh f(\bfx).
\end{equation}
\end{lemma}

Let us interpret the $S$-difference operator through a simple example. For $n=3$ and $S=\{1,2\}$, we have
$$
\Delta_{\bfh}^Sf(\bfx)=f(x_1+h_1,x_2+h_2,x_3)-f(x_1+h_1,x_2,x_3)-f(x_1,x_2+h_2,x_3)+f(x_1,x_2,x_3).
$$
In complete analogy with the discrete concept of marginal interaction among players in a coalition $S\subseteq N$ (see \cite[{\S}2]{GraMarRou00}), the value
$\Delta_{\bfh}^Sf(\bfx)$ can be interpreted as the \emph{marginal interaction} among variables $x_i$ ($i\in S$) at $\bfx$ with respect to the increases $h_i$ for $i\in S$.

Setting $\bfh=\bfy-\bfx$ in the example above, we obtain
$$
\Delta_{\bfy-\bfx}^Sf(\bfx)=f(y_1,y_2,x_3)-f(y_1,x_2,x_3)-f(x_1,y_2,x_3)+f(x_1,x_2,x_3).
$$
If $x_i\leqslant y_i$ for every $i\in S$, then $\Delta_{\bfy-\bfx}^Sf(\bfx)$ is naturally called the \emph{$f$-volume} of the box $\prod_{i\in
S}[x_i,y_i]$. The following straightforward lemma shows that, when $f=v_S$, $\Delta_{\bfy-\bfx}^Sf(\bfx)$ is exactly the volume of the box $\prod_{i\in S}[x_i,y_i]$.

\begin{lemma}\label{lemmadelta}
For every $S\subseteq N$, we have
$
\Delta^S_{\bfy-\bfx}v_S(\bfx)=\prod_{i\in S}(y_i-x_i).
$
\end{lemma}

In the remaining part of this paper, the notation $\bfy_S\in [\bfx_S,\mathbf{1}]$ means that $y_i\in [x_i,1]$ for every $i\in S$.

\begin{theorem}\label{thm:Labreuche}
For every $f\in L^2(\I^n)$ and every $S\subseteq N$, we have
\begin{equation}\label{formuLab}
\mathcal{I}(f,S) = \frac 1{\mu(S)} \int_{\bfx\in\I^n}\int_{\bfy_S\in [\bfx_S,\mathbf{1}]}\Delta^S_{\bfy-\bfx}f(\bfx)\, d\bfy_S\, d\bfx,
\end{equation}
where
$$
\mu(S)=\int_{\bfx\in\I^n}\int_{\bfy_S\in [\bfx_S,\mathbf{1}]} \Delta^S_{\bfy-\bfx}v_S(\bfx)\, d\bfy_S\, d\bfx =6^{-|S|}.
$$
\end{theorem}

\begin{proof}
Since the result is trivial if $S=\varnothing$, we can assume that $S\neq\varnothing$. We first observe that the value of $\mu(S)$ immediately follows from Lemma~\ref{lemmadelta}. Then, for any $T\subseteq N$ and any $i\in T$, we have
\begin{equation}\label{eq:Int11}
\int_0^1\int_{x_i}^{1}E^T_{\bfy-\bfx}f(\bfx)\, dy_i\, dx_i = \int_0^1y_i\, E^T_{\bfy-\bfx}f(\bfx)\, dy_i = \int_0^1x_i\,E^{T\setminus\{i\}}_{\bfy-\bfx}f(\bfx) \, dx_i,
\end{equation}
where the first equality is obtained by permuting the integrals and the second equality by replacing the integration variable $y_i$ with $x_i$. Moreover, we have immediately
\begin{equation}\label{eq:Int12}
\int_0^1\int_{x_i}^1f(\mathbf{x})\, dy_i\, dx_i = \int_0^1(1-x_i)\, f(\mathbf{x})\, dx_i.
\end{equation}

Using (\ref{eq:Int10}) and repeated applications of (\ref{eq:Int11}) and (\ref{eq:Int12}), we finally obtain
\begin{multline*}
\int_{\bfx\in\I^n}\int_{\bfy_S\in [\bfx_S,\mathbf{1}]}\Delta^S_{\bfy-\bfx}f(\bfx)\, d\bfy_S\, d\bfx\\
= \sum_{T\subseteq S}(-1)^{|S|-|T|}\, \int_{\bfx\in\I^n}\int_{\bfy_S\in [\bfx_S,\mathbf{1}]}E^T_{\bfy-\bfx}f(\bfx)\, d\bfy_S\, d\bfx\\
= \sum_{T\subseteq S}(-1)^{|S|-|T|}\, \int_{\I^n}\prod_{i\in T}x_i\, \prod_{i\in S\setminus T}(1-x_i)\, f(\bfx)\, d\bfx\\
= \int_{\I^n}\prod_{i\in S}(2x_i-1)\, f(\bfx)\, d\bfx ~=~ 6^{-|S|}\,\mathcal{I}(f,S),
\end{multline*}
which completes the proof.
\end{proof}

\begin{remark}\label{rem:gdfg45}
\begin{enumerate}
\item[(a)] By Lemma \ref{lemmadelta}, we see that $\mathcal{I}(f,S)$ can be interpreted as the average $f$-volume of the box $\prod_{i\in S}[x_i,y_i]$ divided by its average volume, when $\bfx$ and $\bfy_S$ are chosen at random with the uniform distribution.

\item[(b)] As already mentioned in Remark~\ref{rem:11}$(a)$, Theorem~\ref{thm:Labreuche} appears as a natural generalization of formula (\ref{DiscBanzhaf}) (similarly to Proposition \ref{Beta1}) in the sense that the marginal interaction $\Delta_{\bfh}^Sf(\bfx)$ at $\bfx$ is averaged over the whole domain $\I^n$ (instead of its vertices).

\item[(c)] We note a strong analogy between formula (\ref{formuLab}) and the overall importance index defined by Grabisch and Labreuche in \cite[Theorem~1]{GraLab01}. Indeed, up to the normalization constant, this importance index is obtained by replacing in formula (\ref{formuLab}) the operator $\Delta^S_{\bfy-\bfx}$ by $E^S_{\bfy-\bfx}-I$. Moreover, when $S$ is a singleton, both operators coincide and so do the normalization constants.
\end{enumerate}
\end{remark}

As an immediate consequence of Theorem~\ref{thm:Labreuche}, we have the following interpretation of the index $\mathcal{I}$ as an expected value of the difference quotients of its argument with respect to some probability distribution.

\begin{corollary}\label{cor:sfadfsd}
For every $f\in L^2(\I^n)$ and every $S\subseteq N$, we have
\[
\mathcal{I}(f,S) = \int_{\bfx\in\I^n}\int_{\bfy_S\in [\bfx_S,\mathbf{1}]}p_S(\bfx,\bfy_S)\, Q^S_{\bfy-\bfx}f(\bfx)\, d\bfy_S\, d\bfx,
\]
where the function $p_S(\bfx,\bfy_S)=6^{|S|}\prod_{i\in S}(y_i-x_i)$ defines a probability density function on the set $\{(\bfx,\bfy_S):\bfx\in\I^n, \bfy_S\in [\bfx_S,\boldsymbol{1}]\}$.
\end{corollary}

Let us now analyze the behavior of the interaction index $\mathcal{I}$ on some special classes of functions. The following properties generalize in a very natural way to our setting the behavior of the Banzhaf interaction index $I_B$ with respect to the presence of null players and dummy coalitions.

Recall that a \emph{null player} in a game (or a set function) $v\in\mathcal{G}^N$ is a player $i\in N$ such that $v(T\cup\{i\})=v(T)$ for every $T\subseteq N\setminus\{i\}$. Equivalently, the corresponding pseudo-Boolean function $f\colon\{0,1\}^n\to\R$, given by (\ref{eq:pBfPF}), is independent of $x_i$. The notion of null player for games is then naturally extended through the notion of ineffective variables for functions in $F(\I^n)$ as follows. A variable $x_i$ $(i\in N)$ is said to be \emph{ineffective} for a function $f$ in $F(\I^n)$ if $f(\bfx)=E^{\{i\}}_{-\bfx}f(\bfx)$ for every $\bfx\in\I^n$, or equivalently, if $\Delta^{\{i\}}_{\bfy-\bfx}f(\bfx)=0$ for every $\bfx,\bfy\in\I^n$.

Define $I_f=\{i\in N : \mbox{$x_i$ ineffective for $f$}\}$. From either (\ref{eq:GenBI1}) or (\ref{formuLab}), we immediately derive the following result, which states that any combination of variables containing at least one ineffective variable for a function $f\in L^2(\I^n)$ has necessarily a zero interaction.

\begin{proposition}\label{ineffective}
For every $f\in L^2(\I^n)$ and every $S\subseteq N$ such that $S\cap I_f\neq\varnothing$, we have $\mathcal{I}(f,S) = 0$.
\end{proposition}

We say that a coalition $S\subseteq N$ is \emph{dummy} in a game (or a set function) $v\in\mathcal{G}^N$ if $v(R\cup T)=v(R)+v(T)-v(\varnothing)$ for every $R\subseteq S$ and every $T\subseteq N\setminus S$. This means that $\{S,N\setminus S\}$ forms a partition of $N$ such that, for every coalition $K\subseteq N$, the relative worth $v(K)-v(\varnothing)$ is the sum of the relative worths of its intersections with $S$ and $N\setminus S$. It follows that a coalition $S$ and its complement $N\setminus S$ are simultaneously dummy in any game $v\in\mathcal{G}^N$.

We propose the following extension of this concept.

\begin{definition}
We say that a subset $S\subseteq N$ is \emph{dummy} for a function $f\in F(\I^n)$ if $f(\bfx)=E_{-\bfx}^Sf(\bfx)+E_{-\bfx}^{N\setminus
S}f(\bfx)-f(\mathbf{0})$ for every $\bfx\in\I^n$.
\end{definition}

The following proposition gives an immediate interpretation of this definition.

\begin{proposition}\label{dummy}
A subset $S\subseteq N$ is dummy for a function $f\in F(\I^n)$ if and only if there exist functions $f_S,f_{N\setminus S}\in F(\I^n)$ such that $I_{f_S}\supseteq N\setminus S$, $I_{f_{N\setminus S}}\supseteq S$ and $f=f_S+f_{N\setminus S}$.
\end{proposition}
\begin{proof}
 For the necessity, just set $f_S(\bfx)=E_{-\bfx}^{N\setminus
S}f(\bfx)-f(\mathbf{0})$ and $f_{N\setminus S}=f-f_S$. The sufficiency can be checked directly.
\end{proof}

The following result expresses the natural idea that the interaction for subsets that are properly partitioned by a dummy subset must be zero. It is an immediate consequence of Propositions~\ref{prop:lin56}, \ref{ineffective}, and \ref{dummy}.

\begin{proposition}
For every $f\in L^2(\I^n)$, every nonempty subset $S\subseteq N$ that is dummy for $f$, and every subset $K\subseteq N$ such that $K\cap S\not =\varnothing$ and $K\setminus S\not=\varnothing$, we have $\mathcal{I}(f,K) = 0$.
\end{proposition}

Another immediate consequence of Proposition~\ref{ineffective} is that additive functions have zero interaction indexes for $s$-subsets with $s\geqslant 2$. This fact can be straightforwardly extended to the class of $k$-additive functions as follows.

\begin{definition}
A function $f\in L^2(\I^n)$ is said to be \emph{$k$-additive} for some $k\in\{1,\ldots,n\}$ if there exists a family of functions $\{f_R\in L^2(\I^n) : R\subseteq N,\, |R|\leqslant k\}$ satisfying $I_{f_R}\supseteq N\setminus R$ such that $f=\sum_{R}f_R$.
\end{definition}

\begin{corollary}
Let $f=\sum_{R}f_R\in L^2(\I^n)$ be a $k$-additive function and let $S\subseteq N$. We have $\mathcal{I}(f,S)=0$ if $|S|>k$ and $\mathcal{I}(f,S)=\mathcal{I}(f_S,S)$ if $|S|=k$.
\end{corollary}

Let us now introduce the concept of $S$-increasing monotonicity by refining the classical concept of $n$-increasing monotonicity for functions of $n$ variables (see for instance \cite[p.~43]{Nel06}).

\begin{definition}
Let $S\subseteq N$. We say that a function $f\in F(\I^n)$ is \emph{$S$-increasing} if $\Delta^S_{\bfy-\bfx}f(\bfx)\geqslant 0$ for all $\bfx,\bfy\in \I^n$ such that $\bfx\leqslant \bfy$.
\end{definition}

The following result then follows immediately from Theorem~\ref{thm:Labreuche}.
\begin{proposition}
If $f\in L^2(\I^n)$ is $S$-increasing for some $S\subseteq N$, then ${\mathcal I}(f,S)\geqslant 0$.
\end{proposition}

We end this section by analyzing the behavior of the index ${\mathcal I}$ with respect to dualization, which is a standard concept for instance in aggregation function theory (see \cite[p.~48]{GraMarMesPap09}). The \emph{dual} of a function $f\in F(\I^n)$ is the function $f^d\in F(\I^n)$ defined by $f^d(\bfx)=1- f(\boldsymbol{1}_N-\bfx)$. A function $f\in F(\I^n)$ is said to be \emph{self-dual} if $f^d=f$. By using the change of variables theorem, we immediately derive the following result.

\begin{proposition}\label{prop:jkj945}
For every $f\in L^2(\I^n)$ and every nonempty $S\subseteq N$, we have $\mathcal{I}(f^d,S)=(-1)^{|S|+1}\mathcal{I}(f,S)$. Moreover, $\mathcal{I}(f^d,\varnothing)=1-\mathcal{I}(f,\varnothing)$. In particular, if $f$ is self-dual, then $\mathcal{I}(f,\varnothing)=1/2$ and $\mathcal{I}(f,S)=0$ whenever $|S|$ is even.
\end{proposition}

\begin{remark}
Given $f\in L^2(\I^n)$, we define the \emph{self-dual} and \emph{anti-self-dual} parts of $f$ by $f^s=(f+f^d)/2$ and $f^a=(f-f^d)/2$, respectively. It follows from Proposition~\ref{prop:jkj945} that, for every nonempty $S\subseteq N$, we have $\mathcal{I}(f,S)=\mathcal{I}(f^a,S)$ if $|S|$ is even, and $\mathcal{I}(f,S)=\mathcal{I}(f^s,S)$ if $|S|$ is odd.
\end{remark}

\section{Applications to aggregation function theory}

When we need to summarize, fuse, or merge a set of values into a single one, we usually make use of a so-called aggregation function, e.g., a mean or an averaging function. Various aggregation functions have been proposed thus far in the literature, thus giving rise to the growing theory of aggregation which proposes, analyzes, and characterizes aggregation function classes. For recent references, see Beliakov et al.~\cite{BelPraCal07} and Grabisch et al.~\cite{GraMarMesPap09}.

In this context it is often useful to analyze the general behavior of a given aggregation function $f$ with respect its variables. The index $\mathcal{I}$ then offers a good solution to the problems of (i) determining which variables have the greatest influence over $f$ and (ii) measuring how the variables interact within $f$.

In this section we first compute explicit expressions of the interaction index for the discrete Choquet integral, a noteworthy aggregation function which has been widely investigated due to its many applications for instance in decision making (see for instance \cite{GraMurSug00}). Then we proceed similarly for the class of pseudo-multilinear polynomials, which includes the multiplicative functions and, in particular, the weighted geometric means. Finally, we introduce a normalized version of the index to compare interactions from different functions and compute the coefficient of determination of the multilinear approximations.


\subsection{Discrete Choquet integrals}

A \emph{discrete Choquet integral} is a function $f\in F(\I^n)$ of the form
\begin{equation}\label{Lovasz}
f(\bfx)=\sum_{T\subseteq N} a(T)\,\min_{i\in T}x_i,
\end{equation}
where the set function $a\colon 2^N\to\R$ is nondecreasing with respect to set inclusion and such that $a(\varnothing)=0$ and $\sum_{S\subseteq N}a(S)=1$.\footnote{Whether the conditions on the set function $a$ are assumed or not, the function given in (\ref{Lovasz}) is also called the \emph{Lov\'asz extension} of the pseudo-Boolean function $f|_{\{0,1\}^n}$.} For general background, see for instance \cite[Section~5.4]{GraMarMesPap09}.

The following proposition yields an explicit expression of the interaction index for the class of discrete Choquet integrals. We first consider a lemma and recall that the \emph{beta function} is defined, for any integers $p,q>0$, by
$$
B(p,q)=\int_0^1t^{p-1}(1-t)^{q-1}\, dt = \frac{(p-1)!(q-1)!}{(p+q-1)!}\, .
$$

\begin{lemma}
We have
$$
\int_{[0,1]^n}\min_{i\in T}x_i\prod_{i\in S}\Big(x_i-\frac 12\Big)\, d\bfx =
\begin{cases}
2^{-|S|} \, B(|S|+1,|T|+1), & \mbox{if $S\subseteq T$},\\
0, & \mbox{otherwise}.
\end{cases}
$$
\end{lemma}

\begin{proof}
The result is trivial if $S\nsubseteq T$. Thus, we assume that $S\subseteq T$. Assume also without loss of generality that $T\neq\varnothing$.

For distinct real numbers $x_1,\ldots,x_n$, we have
$$
\min_{i\in T}x_i = \sum_{j\in T}x_j\prod_{i\in T\setminus\{j\}}H(x_i-x_j),
$$
where $H\colon\R\to\R$ is the Heaviside step function ($H(x)=1$ if $x\geqslant 0$ and $0$ otherwise).

Therefore, we have
\begin{multline*}
\int_{[0,1]^n}\min_{i\in T}x_i\prod_{i\in S}\Big(x_i-\frac 12\Big)\, d\bfx
= \int_{[0,1]^{|T|}}\min_{i\in T}x_i\prod_{i\in S}\Big(x_i-\frac 12\Big)\, \prod_{i\in T}dx_i\\
= \sum_{j\in T} \int_0^1\bigg(\int_{[x_j,1]^{|T|-1}}\prod_{i\in S}\Big(x_i-\frac 12\Big)\prod_{i\in T\setminus\{j\}}dx_i\bigg)\, x_j\, dx_j\\
= \sum_{j\in S}\int_0^1\Big(x_j-\frac 12\Big)\Big(\frac{x_j(1-x_j)}{2}\Big)^{|S|-1}(1-x_j)^{|T|-|S|} x_j\, dx_j\\
+ \sum_{j\in T\setminus S}\int_0^1\Big(\frac{x_j(1-x_j)}{2}\Big)^{|S|}(1-x_j)^{|T|-|S|-1}\, x_j\, dx_j\\
= -2^{-|S|}\int_0^1 x_j\,\frac{d}{dx_j}\Big((1-x_j)^{|T|}x_j^{|S|}\Big)\, dx_j.
\end{multline*}
We then conclude by calculating this latter integral by parts.
\end{proof}

\begin{proposition}\label{Lovasz1}
If $f\in F(\I^n)$ is of the form (\ref{Lovasz}), then we have
$$
\mathcal{I}(f,S)=6^{|S|}\sum_{T\supseteq S}a(T)\, B(|S|+1,|T|+1).
$$
\end{proposition}

\begin{remark}
The map $a\mapsto \mathcal{I}(f,S)=6^{|S|}\sum_{T\supseteq S}a(T)\, B(|S|+1,|T|+1)$ defines an interaction index, in the sense of \cite{FujKojMar06}, that is not a probabilistic index (see \cite[Section~3.3]{FujKojMar06}).
However, if we normalize this interaction index (with respect to $|S|$) to get a probabilistic index, we actually divide $\mathcal{I}(f,S)$ by $6^{|S|}B(|S|+1,|S|+1)$ and retrieve the index $I_M$ defined in \cite{MarMat08}.
\end{remark}

\subsection{Pseudo-multilinear polynomials}

We now derive an explicit expression of the index $\mathcal{I}$ for the class of pseudo-multilinear polynomials, that is, the class of multilinear polynomials with transformed variables.

\begin{definition}
We say that a function $f\in L^2(\I^n)$ is a \emph{pseudo-multilinear polynomial} if there exists a multilinear polynomial $g\in F(\R^n)$ and $n$ unary functions $\varphi_1,\ldots,\varphi_n\in L^2(\I)$ such that $f(\bfx)=g(\varphi_1(x_1),\ldots,\varphi_n(x_n))$ for every $\bfx=(x_1,\ldots,x_n)\in\I^n$.
\end{definition}

Using expression (\ref{eq:GinVK}) of multilinear polynomials, we immediately see that any pseudo-multilinear polynomial $f\in L^2(\I^n)$ can be written in the form
\[
f(\bfx)=\sum_{T\subseteq N}a(T)\prod_{i\in T}\varphi_i(x_i).
\]

The following result yields an explicit expression of the interaction index for this function in terms of the interaction indexes for the unary functions $\varphi_1,\ldots,\varphi_n$.

\begin{proposition}\label{prop:dg45456}
For every pseudo-multilinear polynomial $f\in L^2(\I^n)$ and every $S\subseteq N$, we have
\[
\mathcal{I}(f,S)=\sum_{T\supseteq S}a(T)\prod_{i\in T\setminus S}\mathcal{I}(\varphi_i,\varnothing)\prod_{i\in S}\mathcal{I}(\varphi_i,\{i\}).
\]
\end{proposition}

\begin{proof}
By linearity of $\mathcal{I}$, we only have to compute $\mathcal{I}(\prod_{i\in T}\varphi_i,S)$. It is zero if $S\not\subseteq T$ by Proposition \ref{ineffective}. If $S\subseteq T$, we simply use (\ref{eq:GenBI1}) and compute the integrals separately.
\end{proof}

\begin{remark}
Proposition~\ref{prop:dg45456} can actually be easily extended to functions of the form
$$
f(\bfx)=\sum_{T\subseteq N}a(T)\prod_{i\in T}\varphi^T_i(x_i),
$$
where $\varphi^T_i\in L^2(\I)$ for $i=1,\ldots,n$ and $T\subseteq N$.
\end{remark}

An interesting subclass of pseudo-multilinear polynomials is the class of multiplicative functions, that is, functions of the form $f(\bfx)=\prod_{i=1}^n \varphi_i(x_i)$, where $\varphi_1,\ldots,\varphi_n\in L^2(\I)$. For every multiplicative function $f\in L^2(\I^n)$ and every $S\subseteq N$, assuming $\mathcal{I}(f,\varnothing)\not=0$, the ratio $\mathcal{I}(f,S)/\mathcal{I}(f,\varnothing)$ is also multiplicative in the sense that
\begin{equation}\label{eq:gekelk4}
\frac{\mathcal{I}(f,S)}{\mathcal{I}(f,\varnothing)}=\prod_{i\in S} \frac{\mathcal{I}(\varphi_i,\{i\})}{\mathcal{I}(\varphi_i,\varnothing)}.
\end{equation}
Combining this with (\ref{eq:Approx}) and the identity $\sum_{T\subseteq N}\prod_{i\in T} z_i=\prod_{i\in N}(1+z_i)$, we can write the best
$n$th approximation of $f$ as
$$
f_n(\bfx)=\mathcal{I}(f,\varnothing) \, \prod_{i\in
N}\Big(1+\frac{\mathcal{I}(\varphi_i,\{i\})}{\mathcal{I}(\varphi_i,\varnothing)}\,\big(x_i-\frac 12\big)\Big).
$$

\subsection{Normalized index and coefficients of determination}

Just as for interaction indexes introduced in game theory \cite{GraRou99} and the importance index defined by Grabisch and Labreuche \cite{GraLab01}, the interaction index $\mathcal{I}$ is a linear map. This implies that it cannot be considered as an absolute interaction index but rather as a relative index constructed to assess and compare interactions for a \emph{given} function.

If we want to compare interactions for \emph{different} functions, we need to consider an absolute (normalized) interaction index. Such an index is actually easy to define if we use the following probabilistic viewpoint: considering the unit cube $\I^n$ as a probability space with respect to the Lebesgue measure, we see that, for a nonempty subset $S\subseteq N$, the index $\mathcal{I}(f,S)$ is actually the covariance of the random variables $f$ and $12^{|S|/2}w_S$. It is then natural to consider the Pearson (or correlation) coefficient instead of the covariance.

\begin{definition}
The \emph{normalized interaction index} is the mapping $$r\colon \{f\in L^2(\I^n):\sigma(f)\neq 0\}\times (2^N\setminus\{\varnothing\})\to\R$$ defined by
$$
r(f,S)=\frac{\mathcal{I}(f,S)}{12^{|S|/2}\,\sigma(f)}=\Big\langle\frac{f-E(f)}{\sigma(f)}\, ,w_S\Big\rangle\, ,
$$
where $E(f)$ and $\sigma(f)$ are the expectation and the standard deviation of $f$, respectively, when $f$ is regarded as a random variable.
\end{definition}

From this definition it follows that $-1\leqslant r(f,S)\leqslant 1$. Moreover, this index remains unchanged under interval scale transformations, that is, $r(af+b,S)=r(f,S)$ for all $a>0$ and $b\in\R$. Note also that the normalized indexes for a function $f\in L^2(\I^n)$ and its dual $f^d$ are linked by $r(f^d,S)=(-1)^{|S|+1}\, r(f,S)$, where $S\neq\varnothing$.

Let us examine on a few examples the behavior of the normalized importance index $r(f,\{i\})$ of a variable $x_i$:
\begin{itemize}
\item For the arithmetic mean $f(\bfx)=\frac{1}{n}\sum_{i=1}^nx_i$, we have $\sigma(f)=(12 n)^{1/2}$, $\mathcal{I}(f,\{i\})=1/n$ for all $i\in N$, and hence $r(f,\{i\})=1/\sqrt{n}$.

\item For the minimum function $f(\bfx)=\min_{i\in N}x_i$, we have
$$
\sigma(f)=\frac{\sqrt{n}}{(n+1)\sqrt{n+2}}
$$
(see \cite[Lemma 6]{MarMat08}). By Proposition \ref{Lovasz1}, we then have
$$
r(f,\{i\})=\frac{\sqrt{3}}{\sqrt{n (n+2)}}
$$
for every $i\in N$.
By duality, the same result holds for the maximum function $f^d(\bfx)=\max_{i\in N}x_i$. From this fact, we measure the intuitive fact that the overall importance of a given variable is greater in the arithmetic mean than in the minimum and the maximum functions.

\item Consider the weighted geometric mean $f(\bfx)=\prod_{i=1}^n x_i^{c_i}$, where $c_1,\ldots,c_n\geqslant 0$ and $\sum_{i=1}^nc_i=1$. Using (\ref{eq:gekelk4}), for every nonempty subset $S\subseteq N$, we have
$$
\mathcal{I}(f,S)=\prod_{i\in N}\frac{1}{c_i+1}\prod_{i\in S}\frac{6 c_i}{c_i+2}.
$$
In the special case of the symmetric geometric mean function, we have
$$
r(f,\{i\})=\frac{\sqrt{3}}{2n+1}\Big(\Big(\frac{(n+1)^2}{n(n+2)}\Big)^n-1\Big)^{-1/2}.
$$
Here again, we can show that the importance of variables in the arithmetic mean is greater than the importance of variables in the geometric mean function.
\end{itemize}

The normalized index is also useful to compute the \emph{coefficient of determination} of the best $k$th approximation of $f$.\footnote{This coefficient actually measures the \emph{goodness of fit} of the multilinear model.} Assuming that $\sigma(f)\neq 0$, this coefficient is given by
$$
R^2_k(f)=\frac{\sigma^2(f_k)}{\sigma^2(f)}.
$$
Since $E(f_k)=\mathcal{I}(f_k,\varnothing)=\mathcal{I}(f,\varnothing)=E(f)$ (see Proposition~\ref{prop:Proj-IB}), by (\ref{eq:akS2}), we obtain
$$
\sigma^2(f_k)=\|f_k-E(f_k)\|^2=\bigg\|\sum_{\textstyle{T\subseteq N\atop 1\leqslant |T|\leqslant k}}\langle f,w_T\rangle\, w_T\bigg\|^2=\sum_{\textstyle{T\subseteq N\atop 1\leqslant |T|\leqslant k}}\langle f,w_T\rangle^2
$$
and hence
\begin{equation}\label{eq:dfs4df56}
R^2_k(f)=\sum_{\textstyle{T\subseteq N\atop 1\leqslant |T|\leqslant k}}r(f,T)^2.
\end{equation}

\begin{remark}
The coefficient of determination explains why the normalized importance of each variable in the arithmetic mean is greater than that in the minimum function, the maximum function, and the geometric mean function. Indeed, if $f$ is a symmetric function, then $r(f,\{i\})=r(f,\{j\})$ for every $i,j\in N$ and, since $R^2_1(f)\leqslant 1$, by (\ref{eq:dfs4df56}) we have $r(f,\{i\})\leqslant 1/\sqrt n$.
\end{remark}

\section*{Acknowledgments}

The authors wish to thank Michel Beine, Miguel Couceiro, Paul G\'erard, and Samuel Nicolay for fruitful discussions. This research is supported by the internal research project F1R-MTH-PUL-09MRDO of the University of Luxembourg.


\begin{thebibliography}{10}

\bibitem{AikWes91}
L.~S. Aiken and S.~G. West.
\newblock {\em {Multiple Regression: Testing and Interpreting Interactions}}.
\newblock {Newbury Park-London-New Delhi: Sage Publications}, 1991.

\bibitem{Aub81}
J.-P. Aubin.
\newblock Cooperative fuzzy games.
\newblock {\em Math. Oper. Res.}, 6(1):1--13, 1981.

\bibitem{Ban65}
J.~F. Banzhaf.
\newblock Weighted voting doesn't work: A mathematical analysis.
\newblock {\em Rutgers Law Review}, 19:317--343, 1965.

\bibitem{BelPraCal07}
G.~Beliakov, A.~Pradera, and T.~Calvo.
\newblock {\em Aggregation Functions: A Guide for Practitioners}.
\newblock Studies in Fuziness and Soft Computing. Springer, Berlin, 2007.

\bibitem{BenLin90}
M.~Ben-Or and N.~Linial.
\newblock Collective coin flipping.
\newblock In {\em Randomness and Computation}, pages 91--115. Academic Press,
  New York, 1990.

\bibitem{BouKahKalKatLin92}
J.~Bourgain, J.~Kahn, G.~Kalai, Y.~Katznelson, and N.~Linial.
\newblock {The influence of variables in product spaces}.
\newblock {\em Isr. J. Math.}, 77(1-2):55--64, 1992.

\bibitem{DubSha79}
P.~Dubey and L.~S. Shapley.
\newblock {Mathematical properties of the {B}anzhaf power index}.
\newblock {\em Math. Oper. Res.}, 4:99--131, 1979.

\bibitem{FujKojMar06}
K.~Fujimoto, I.~Kojadinovic, and J.-L. Marichal.
\newblock Axiomatic characterizations of probabilistic and
  cardinal-probabilistic interaction indices.
\newblock {\em Games Econom. Behav.}, 55(1):72--99, 2006.

\bibitem{GraLab01}
M.~Grabisch and C.~Labreuche.
\newblock How to improve acts: An alternative representation of the importance
  of criteria in {MCDM}.
\newblock {\em Int. J. Uncertain. Fuzziness Knowl.-Based Syst.}, 9(2):145--157,
  2001.

\bibitem{GraMarMesPap09}
M.~Grabisch, J.-L. Marichal, R.~Mesiar, and E.~Pap.
\newblock {\em Aggregation functions}, volume 127 of {\em Encyclopedia of
  Mathematics and its Applications}.
\newblock Cambridge University Press, Cambridge, 2009.

\bibitem{GraMarRou00}
M.~Grabisch, J.-L. Marichal, and M.~Roubens.
\newblock Equivalent representations of set functions.
\newblock {\em Math. Oper. Res.}, 25(2):157--178, 2000.

\bibitem{GraMurSug00}
M.~Grabisch, T.~Murofushi, and M.~Sugeno, editors.
\newblock {\em Fuzzy measures and integrals - Theory and applications},
  volume~40 of {\em Studies in Fuzziness and Soft Computing}.
\newblock Physica-Verlag, Heidelberg, 2000.

\bibitem{GraRou99}
M.~Grabisch and M.~Roubens.
\newblock {An axiomatic approach to the concept of interaction among players in
  cooperative games}.
\newblock {\em Int. J. Game Theory}, 28(4):547--565, 1999.

\bibitem{HamHol92}
P.~Hammer and R.~Holzman.
\newblock {Approximations of pseudo-Boolean functions; applications to game
  theory}.
\newblock {\em Z. Oper. Res.}, 36(1):3--21, 1992.

\bibitem{HamRud68}
P.~Hammer and S.~Rudeanu.
\newblock {\em {Boolean methods in operations research and related areas}}.
\newblock {Berlin-Heidelberg-New York: Springer-Verlag}, 1968.

\bibitem{KahKalLin88}
J.~Kahn, G.~Kalai, and N.~Linial.
\newblock The influence of variables on {B}oolean functions.
\newblock In {\em Proc.\ 29th Annual Symposium on Foundations of Computational
  Science}, pages 68--80. {Computer Society Press}, 1988.

\bibitem{Mar00}
J.-L. Marichal.
\newblock The influence of variables on pseudo-{B}oolean functions with
  applications to game theory and multicriteria decision making.
\newblock {\em Discrete Appl. Math.}, 107(1-3):139--164, 2000.

\bibitem{MarKojFuj07}
J.-L. Marichal, I.~Kojadinovic, and K.~Fujimoto.
\newblock Axiomatic characterizations of generalized values.
\newblock {\em Discrete Applied Mathematics}, 155(1):26--43, 2007.

\bibitem{MarMat08}
J.-L. Marichal and P.~Mathonet.
\newblock Approximations of {L}ov\'asz extensions and their induced interaction
  index.
\newblock {\em Discrete Appl. Math.}, 156(1):11--24, 2008.

\bibitem{Nel06}
R.~B. Nelsen.
\newblock {\em An introduction to copulas}.
\newblock Springer Series in Statistics. Springer, New York, second edition,
  2006.

\bibitem{Sha53}
L.~Shapley.
\newblock {A value for $n$-person games}.
\newblock In {\em {Contributions to the Theory of Games II (Annals of
  Mathematics Studies 28)}}, pages 307--317. {Princeton University Press},
  1953.

\end{thebibliography}

\end{document}